\documentclass[12pt]{amsart} 
\usepackage[top=1in, bottom=1in, left=1in, right=1in]{geometry}

\usepackage{times}
\usepackage[table,dvipsnames]{xcolor}

\PassOptionsToPackage{hyphens}{url}\usepackage[colorlinks=true, linkcolor=NavyBlue, citecolor=teal, urlcolor=gray]{hyperref}   
\numberwithin{equation}{section}

\usepackage{scrextend}
\usepackage{amssymb,mathrsfs,amsmath}
\usepackage{mathtools}
\usepackage{bbm}

\theoremstyle{plain}
\newtheorem{thm}{Theorem}
\newtheorem{prop}{Proposition}[section]
\newtheorem{lem}[prop]{Lemma}
\newtheorem{claim}[prop]{Claim}

\theoremstyle{definition}
\newtheorem{dfn}[prop]{Definition}

\theoremstyle{remark}
\newtheorem*{rem*}{Remark}
\newtheorem*{rems*}{Remarks}

\newcommand{\eq}[1]{ \begin{equation}\begin{split}
			#1 \end{split}
\end{equation} }

\newcommand{\als}[1]{\begin{align*} #1 \end{align*} }

\newcommand{\N}{\mathbb{N}}
\newcommand{\Z}{\mathbb{Z}}

\newcommand{\R}{\mathbb{R}}
\newcommand{\C}{\mathbb{C}}

\newcommand{\E}{\mathbb{E}}
\renewcommand{\P}{\mathbb{P}}

\newcommand{\CA}{\mathcal{A}}
\newcommand{\CB}{\mathcal{B}}
\newcommand{\CC}{\mathcal{C}}

\newcommand{\ds}{\displaystyle}

\newcommand{\bv}\boldsymbol{}

\newcommand{\eps}\varepsilon

\renewcommand{\le}{\leqslant}
\renewcommand{\geq}{\geqslant}
\renewcommand{\ge}{\geqslant}

\renewcommand{\mod}[1]{\,(\mathrm{mod}\,#1)}

\newcommand{\bggg}{\bigg}

\begin{document}

\title{On the $j$-th smallest modulus of a covering system with distinct moduli}

\author{Jonah Klein}
\address{D\'epartement de math\'ematiques et de statistique\\
	Universit\'e de Montr\'eal\\
	CP 6128 succ. Centre-Ville\\
	Montr\'eal, QC H3C 3J7\\
	Canada}
\email{{\tt jonah.klein@umontreal.ca}}

\author{Dimitris Koukoulopoulos}
\address{D\'epartement de math\'ematiques et de statistique\\
	Universit\'e de Montr\'eal\\
	CP 6128 succ. Centre-Ville\\
	Montr\'eal, QC H3C 3J7\\
	Canada}
\email{{\tt dimitris.koukoulopoulos@umontreal.ca}}

\author{Simon Lemieux}
\address{D\'epartement de math\'ematiques et de statistique\\
	Universit\'e de Montr\'eal\\
	CP 6128 succ. Centre-Ville\\
	Montr\'eal, QC H3C 3J7\\
	Canada}
\email{{\tt simon.lemieux.6@umontreal.ca}}

%
\date{\today}

\maketitle
\begin{abstract}
Covering systems were introduced by Erd\H{o}s in 1950. In the same article where he introduced them, he asked if the minimum modulus of a covering system with distinct moduli is bounded. In 2015, Hough answered affirmatively this long standing question.  In 2022, Balister, Bollob\'as, Morris, Sahasrabudhe and Tiba gave a simpler and more versatile proof of Hough's result. Building upon their work, we show that there exists some absolute constant $c>0$ such that the $j$-th smallest modulus of a minimal covering system with distinct moduli is $\le \exp(cj^2/\log(j+1))$.
\end{abstract}
\section{Introduction}

A {\it covering system} is a finite set of arithmetic progressions whose union is $\Z$. We say a covering system is {\it minimal} if there exists no proper subset of its arithmetic progressions that also covers $\Z$.

Covering systems were first introduced by Paul Erd\H os in a seminar 1950 article \cite{ErdosIntro}, where he used them to give an answer to a question of Romanoff on integers of the form $2^k+p$ with $p$ a prime and $k$ an integer. More precisely, Erd\H os proved that there exists an arithmetic progression that contains no integers of the above form.

Throughout his life, Erd\H os posed many questions on the properties of covering systems. One of the most famous ones comes from his 1950 article: consider all covering systems with distinct moduli. Is there a uniform bound on the smallest modulus of all these systems? This problem remained open for 65 years, until it was recently solved by Hough \cite{Hough}. In fact, Hough proved that every such system has minimum modulus at most $10^{16}$. In 2022, Balister, Bollob\'as, Morris, Sahasrabudhe and Tiba \cite{Balister} reduced this bound to 616,000. Their method, which they named {\it the distortion method}, is rather versatile -- see \cite{Balister4} for an expository note on it and \cite{Balister2} for another application of it to the so-called {\it Erd\H os--Selfridge problem}.  Using the distortion method, Cummings, Filaseta and Trifonov~\cite{CFT} improved Hough's bound to 118 under the additional assumption that all moduli are square-free. In the present paper, we use this method to demonstrate the following generalization of Hough's theorem.

\begin{thm}\label{thm1}
	There exists an absolute constant $c>0$ such that the $j$-th smallest modulus in a minimal covering system with distinct moduli is $\le \exp(cj^2/\log(j+1))$.
\end{thm}

Except for the distortion method, the other key input to the proof of Theorem \ref{thm1} is the following result of Crittenden and Vanden Eynden~\cite{Crittenden}, originally conjectured by Erd\H os~\cite{Erdosn}: if a set of $n$ arithmetic progressions does not cover $\Z$, then it does not cover the interval $\{1,2,\dots,2^n\}$. Interestingly, a shorter proof of this theorem was recently discovered by Balister, Bollob\'as, Morris, Sahasrabudhe and Tiba \cite{Balister2}.

\begin{rem*} Cummings, Filaseta and Trifonov~\cite{CFT} have obtained a weaker version of Theorem \ref{thm1} independently: they show that the $j$-th smallest modulus  in a minimal covering system with distinct moduli is $\le C_j$ for some unspecified constant $C_j$. Their proof is rather different, showing the existence of $C_j$ in an inductive fashion, as a corollary of Hough's theorem. 
\end{rem*}

We conclude this introductory section with the following result that complements Theorem \ref{thm1}.

\begin{thm}\label{thm:lb} For each $j \geq 5$, there exists a minimal covering system with the following $j$ distinct moduli placed in increasing order:
	\[
	2<2^2<2^3<\dots<2^{j-4}<3\cdot 2^{j-5}<2^{j-3}<3\cdot 2^{j-4}<3\cdot 2^{j-3}.
	\]
\end{thm}

\begin{proof} 
Fix $j\geq 5$, and let
\[
\CC_j=\{1 \mod 2, 2 \mod 4, \dots, 2^{j-4} \mod{2^{j-3}}, A_0, A_1, A_2\},
\]
where $A_k$ is the intersection of the congruence classes $k\mod 3$ and $0\mod{2^{j-5+k}}$.  We claim that $\CC_j$ is a covering system. 

Indeed, for each integer $n$, either $2^{j-3}|n$ or $2^{j-3}\nmid n$. In the former case, let $k\in\{0,1,2\}$ be such that $n\equiv k\mod 3$. In particular, we have $n\in A_k$, and thus $n$ is covered by $\CC_j$. Let us now consider the case when $2^{j-3}\nmid n$. Then, there must exist some $i\in\{1,2\dots,j-3\}$ such that $2^{i-1}|n$ and $2^i\nmid n$. Hence, $n\equiv 2^{i-1}\mod {2^i}$, so that $n$ is covered again by $\CC_j$. 

We have thus proven that $\CC_j$ is a covering system. Clearly, its list of moduli is the one prescribed in the statement of the theorem. Lastly, $\CC_j$ is a minimal covering system: the arithmetic progressions $2^{i-1}\mod{2^i}$ with $i=1,\dots,j-3$ are disjoint and cover exactly the integers not divisible by $2^{j-3}$. On the other hand, the arithmetic progressions $A_0,A_1,A_2$ are disjoint and are all needed to cover the integers in $0\mod{2^{j-3}}$.  This completes the proof of the theorem.
\end{proof}

\section{Outline of the proof of Theorem \ref{thm1}}

Let 
\[
\CC=\{r_1\mod{q_1},r_2\mod{q_2},\dots,r_k\mod{q_k}\}
\] 
be a minimal covering system with $q_1<q_2<\cdots<q_k$, and let $\ell \in\{1,2,\dots,k\}$. In particular, the first $\ell-1$ congruence classes do not cover $\Z$. Now, consider the system
\eq{\label{C_ell}
\CC_\ell=  \big\{r_j-h \mod{q_j} : \ell\le j\le k,\ 0\le h<2^{\ell-1} \big\} ,
}
so that $\CC_1=\CC$. 

\begin{claim}\label{claim1}
	The set $\CC_\ell$ is a covering system for $\ell=1,2,\dots,k$.
\end{claim}

\begin{proof}
Recall that  Crittenden and Vanden Eynden \cite{Crittenden} proved that if $\ell$ progressions cover an interval $I \subset \Z$ of length $\geq 2^\ell$, then they cover $\Z$. Hence, for any fixed $n\in\Z$, we must have 
\[
n+ \{0, \ldots, 2^{\ell-1}-1\} \not\subset \bigcup_{i=1}^{\ell-1}  \big(q_i\Z+r_i\big) \,;
\]
otherwise the progressions on the right-hand side would cover $\Z$, which would contradict the minimality of the covering system $\CC$. Therefore, there is some $h \in \{0,\ldots,2^{\ell-1}-1\}$ such that $n+h \notin q_i\Z+r_i$ for each $i\in\{0,1,\dots,\ell-1\}$. Since $\CC$ is a covering system, there must exist some $j\ge\ell$ such that $n+h$ is covered by $q_j\Z+r_j$. We conclude that $n \in q_j\Z +r_j-h$, and so $n$ is covered by a progression in $C_\ell$.  
\end{proof}

The above claim motivates the following definition:

\begin{dfn}
Let $\CA=\{a_1\mod{d_1},a_2\mod{d_2},\dots,a_n\mod{d_n}\}$ be a set of congruences. We define the {\it multiplicity} of $\CA$ to be the number
\[
m(\CA):= \max_{d\in\N} \#\{1\le j\le n: d_j=d\} .
\]
\end{dfn}

\begin{rem*}
The moduli of a system of congruences $\CA$ are distinct if and only if $m(\CA)=1$. 
\end{rem*}

With this definition in mind, we have that $m(\CC_\ell)=2^{\ell-1}$ for the system of congruences defined in \eqref{C_ell}. Hence, Theorem \ref{thm1} is an immediate corollary of the following result. 

\begin{thm}\label{thm2} Let $\CA$ be a covering system of multiplicity $s$. Then there exists an absolute constant $c>0$ such that its smallest modulus is $\le \exp(c \log^2(s+1)/\log\log(s+2))$. 
\end{thm}

We will prove Theorem \ref{thm2} in the following section by a suitable modification of the distortion method.

\subsection*{Notation}
Given $n\in\N$, we write $P^+(n)$ for its largest prime factor with the convention that $P^+(1)=1$. In addition, we write $\omega(n)$ for the number of distinct prime factors of $n$. 

We adopt the usual asymptotic notation of Vinogradov: given two functions $f,g:X\to\R$ and a set $Y\subseteq X$, we write ``$f(x)\ll g(x)$ for all $x\in Y$'' if there is a constant $c=c(f,g,Y)>0$ such that $|f(x)|\le cg(x)$ for all $x\in Y$. The constant is absolute unless otherwise noted by the presence of a subscript. If $h:X\to\R$ is a third function, we use Landau's notation ``$f=g+O(h)$ on $Y$'' to mean that $|f-g|\ll h$ on $Y$. Typically the set $Y$ is clear from the context and so not stated explicitly.  Finally, if $f\ll g$ and $g\ll f$, we write $f\asymp g$.

\section{Covering systems of bounded multiplicity}\label{sec:proofs}

\subsection{The distortion method}

In this section we slightly modify some of the definitions in \cite{Balister} (see also \cite{Balister4}) to allow for covering systems of multiplicity $>1$.

We start with a finite set of congruences $\CA=\{a_1\mod{d_1},\dots,a_n\mod{d_n}\}$ with $1<d_1\le d_2\le \cdots\le d_n$. The goal is to show that if $d_1$ is large enough in terms of the multiplicity of $\CA$, then $\CA$ is not a covering system. 

Let $Q=[d_1,\dots,d_n]$, and let $p_1<p_2<\cdots<p_J$ be the distinct primes dividing $Q$. We may then write
\[
Q=\prod_{i=1}^J p_i^{\nu_i},
\]
where $\nu_i$ is the $p_i$-adic valuation of $Q$. For $j\in\{1,2,\dots,J\}$, define
\[
Q_j:=\prod_{i=1}^j p_i^{\nu_i}
\qquad\text{and}\qquad
\CB_j:= \bigcup_{\substack{1\le i\le n \\ P^+(d_i)=p_j }} \big\{ a\mod{Q} : a\equiv a_i\mod{d_i} \big\}.
\]

The most crucial definition is that of certain probability measures $\P_0,\P_1,\dots,\P_J$ on $\Z/Q\Z$, which we construct exactly as in \cite{Balister} in terms of some free parameters $\delta_1,\dots,\delta_J\in[0,1/2]$. 

First of all, we set up some notation. Let $\pi_j:\Z/Q\Z\to\Z/Q_j\Z$ be the natural projection for all $j\in\{0,1,\dots,J\}$, where $Q_0=1$. In addition, let
\[
F_j(x):= \{x'\in \Z/Q\Z: \pi_j(x')=\pi_j(x)\} ,
\]
so that $|F_j(x)|=Q/Q_j$. The measure $\P_j$ will be {\it $Q_j$-measurable} by construction, meaning that it will have the property that
\[
\P_j(x) = \P_j(x')\quad\text{whenever}\ \pi_j(x)=\pi_j(x'). 
\]

We are now ready to define the measures $\P_j$. We begin by letting $\P_0$ be the uniform measure on $\Z/Q\Z$. Next, consider $j\in\{1,\dots,J\}$ and suppose we have already constructed $\P_{j-1}$ to be $Q_{j-1}$-measurable. Let
\[
\alpha_j(x) := \frac{|F_{j-1}(x)\cap \CB_j|}{|F_{j-1}(x)|} 
\qquad\text{for all}\ x\in \Z/Q\Z,
\]
and note that $\alpha_j$ is a $Q_{j-1}$-measurable function, meaning that $\alpha_j(x')=\alpha_j(x)$ if $\pi_{j-1}(x')=\pi_{j-1}(x)$. We then define $\P_j$ on the congruence class $x\in\Z/Q\Z$ as follows:
\begin{itemize}
	\item If $\alpha_j(x)<\delta_j$, we let
	\[
	\P_j(x):= \P_{j-1}(x) \cdot \frac{1_{x\notin \CB_j}}{1-\alpha_j(x)} . 
	\]
	\item If $\alpha_j(x)\ge\delta_j$, we let
\[
\P_j(x):=\P_{j-1}(x) \cdot
	\begin{cases} 
		\ds	\frac{\alpha_j(x)-\delta_j}{\alpha_j(x)(1-\delta_j)}
			&\mbox{if $x\in \CB_j$},\\
			\\
		\ds \frac{1}{1-\delta_j} &\mbox{if $x\notin \CB_j$}.
	\end{cases}
\]
\end{itemize}
We may easily check that 
\eq{\label{eq:compatibility}
\P_j(F_{j-1}(x)) = \sum_{x'\in F_{j-1}(x)} \P_j(x')=\P_{j-1}(x)\cdot |F_{j-1}(x)| = \P_{j-1}(F_{j-1}(x)) 
}
for all $x\in\Z/Q\Z$. Since $\P_{j-1}$ is a probability measure on $\Z/Q\Z$, so is $\P_j$. In addition, $\P_j$ is $Q_j$-measurable by construction. We have thus completed the inductive step.

Having defined the measures $\P_j$, we introduce the notation 
\[
\E_j[f]:= \sum_{x\in \Z/Q\Z} f(x)\P_j(x)\quad\text{for all functions}\ f:\Z/Q\Z\to \C.
\]
We then define 
\[
M_j^{(1)}:=\E_{j-1}[\alpha_j] 
\quad\text{and}\quad
M_j^{(2)}:=\E_{j-1}[\alpha_j^2]
\]
for $j=1,\dots,J-1$. 

Given the above notation and definitions, the key result of the distortion method is the following.

\begin{lem}[Theorem 3.1 in \cite{Balister}]\label{lemma1} Assume the above notation. If
	\[
	\sum_{j=1}^J
	\min\bggg\{  M_j^{(1)},\frac{M_j^{(2)}}{4\delta_j(1-\delta_j)}\bggg\}<1 ,
	\]
	then $\CA$ does not cover the integers. 
\end{lem}

\begin{rem*}
	As a matter of fact, Theorem 3.1 in \cite{Balister} also contains a second part about the density of the uncovered set $R$, but we do not need it here.
\end{rem*}

\subsection{Bounding the first and second moment}
We now proceed with bounding $M_j^{(1)}$ and $M_j^{(2)}$. Doing so is the context of Theorem 3.2 in \cite{Balister}, but this result is only valid for systems of congruences of multiplicity $1$. We thus need to generalize it. This is rather straightforward, and we describe how to do it below.

\begin{lem}\label{lem:B in fibers}
Assume the above notation. For $x\in\Z/Q\Z$ and $j\in\{1,2,\dots,J\}$, we have
\[
\alpha_j(x) \le \sum_{r=1}^{\nu_j} \sum_{g|Q_{j-1}}  \sum_{\substack{1\le i\le n \\ d_i=g p_j^r}} \frac{1_{x \subseteq a_i+g\Z}}{p_j^r}. 
\]
\end{lem}

\begin{proof} Note that $|F_{j-1}(x)|=Q/Q_{j-1}$ and that we may write $x=c+Q\Z$ for some $c\in\Z$. Hence,
	\[
	\alpha_j(x) =\frac{|F_{j-1}(x)\cap \CB_j|}{Q/Q_{j-1}}  \le \frac{Q_{j-1}}{Q} \sum_{\substack{1\le i\le n \\ P^+(d_i)=p_j}} \sum_{\substack{a\mod Q \\ a\equiv c\mod{Q_{j-1}} \\ a\equiv a_i\mod{d_i}}}1 ,
	\]
	by the union bound. For each $i$ with $P^+(d_i)=p_j$ we may write uniquely $d_i=gp_j^r$ with $g|Q_{j-1}$ and $1\le r\le\nu_j$. We thus find that
\[ 
\alpha_j(x) \le \frac{Q_{j-1}}{Q_j} \sum_{r=1}^{\nu_j} \sum_{g|Q_{j-1}}  \sum_{\substack{1\le i\le n \\ d_i=g p_j^r}}\sum_{\substack{a\mod Q \\ a\equiv c\mod{Q_{j-1}} \\ a\equiv a_i\mod{d_i}}}1.
\]
For the congruences $a\equiv a_i\mod{d_i}$ and $a\equiv c\mod{Q_{j-1}}$ to be compatible, we must  have $c\equiv a_i\mod{d_i}$ or, equivalently, that $x$ is a subset of $a_i+g\Z$. Under this assumption, $a$ lies in some congruence class mod $Q_{j-1}p_j^r$, so there are $Q/(Q_{j-1}p_j^r)$ choices for $a\mod Q$. This completes the proof of the lemma.
\end{proof}

\begin{lem}\label{lem:moments} 
	Assume the above notation, let $s=m(\CA)$, and let $j\in\{1,2,\dots,J\}$. 
	\begin{enumerate}
	\item If $\delta_i=0$ for $i\in\{1,\dots,j-1\}$, then 
	\[
	M_j^{(1)} \le  s \sum_{\substack{d\ge d_1 \\ P^+(d)=p_j}} \frac{1}{d} .
	\]
	\item We have
	\[
	M_j^{(2)} \ll \frac{s^2(\log p)^6}{p^2}.
	\]
		\end{enumerate}
\end{lem}

\begin{proof} We treat both parts simultaneously for now. Let $k\in\{1,2\}$ and let us write $p=p_j$ for simplicity. By Lemma \ref{lem:B in fibers}, we have 
	\[
	\E_{j-1}[\alpha_j^k] \le \sum_{1\le r_1,\dots,r_k\le \nu_j}\sum_{g_1,\dots,g_k | Q_{j-1} }  \sum_{\substack{1\le i_1,\dots,i_k \le n \\ d_{i_\ell}=g_\ell p^{r_\ell} \ \forall \ell }} \frac{\P_{j-1}\big( \bigcap_{\ell=1}^k (a_{i_\ell}+g_\ell\Z)\big)}{p^{r_1+\cdots+r_k}} .
	\]
	Since $d_i\ge d_1$ for all $i$, we must have $g_\ell p^{r_\ell}\ge d_1$ for all $\ell$. Given $r_1,\dots,r_k$ and $g_1,\dots,g_k$, there are at most $s^k$ choices for $i_1,\dots,i_k$ with $d_{i_\ell}=g_\ell p^{r_\ell}$ (because we have assumed that $\CA$ has multiplicity $s$). For each such choice of $i_1,\dots,i_k$, the Chinese Remainder Theorem implies that the set $\bigcap_{\ell=1}^k (a_{i_\ell}+g_\ell\Z)$ is either empty, or an arithmetic progression with modulus $[g_1,\dots,g_k]$. Hence, Lemma 3.4 in \cite{Balister} implies that
	\eq{\label{eq:AP bound}
	\P_{j-1}\bigg( \bigcap_{\ell=1}^k (a_{i_\ell}+g_\ell\Z)\bigg)
		\le \frac{\prod_{p_i|[g_1,\dots,g_k]}(1-\delta_i)^{-1}}{[g_1,\dots,g_k]},
	}
for each of the $\le s^k$ possible values of $i_1,\dots,i_k$. We thus conclude that
	\[
	\E_{j-1}[\alpha_j^k] \le s^k \sum_{1\le r_1,\dots,r_k\le \nu_j}\sum_{\substack{g_1,\dots,g_k | Q_{j-1} \\  g_\ell p^{r_\ell}\ge d_1\ \forall\ell}}  \frac{\prod_{p_i|[g_1,\dots,g_k]}(1-\delta_i)^{-1}} {[g_1,\dots,g_k]p^{r_1+\cdots+r_k}}  .
	\]
	When $k=1$ and $\delta_i=0$ for all $i<j$, this readily proves part (a) of the lemma. 
	
	Now, let us consider the case when $k=2$ and prove part (b). Here, there are no conditions on the parameters $\delta_i$ except for knowing that $\delta_i\in[0,1/2]$ for all $i$. In particular,  $\prod_{p_i|[g_1,g_2]}(1-\delta_i)^{-1}\le 2^{\omega([g_1,g_2])}$. Therefore,
	\als{
	\E_{j-1}[\alpha_j^2] &\le s^2 \sum_{1\le r_1,r_2\le \nu_j}\sum_{g_1,g_2 | Q_{j-1} }  \frac{2^{\omega([g_1,g_2])}}{[g_1,g_2]p^{r_1+r_2}} 
	\le \frac{s^2}{(p-1)^2}\sum_{g_1,g_2 | Q_{j-1} }  \frac{2^{\omega([g_1,g_2])}}{[g_1,g_2]} .
	}
	The function $\N\ni m\to \#\{(g_1,g_2)\in\N^2: [g_1,g_2]=m\}$ is multiplicative and takes the value $2\nu+1$ on each $\nu$-th prime power. Hence,
	\[
	\sum_{g_1,g_2 | Q_{j-1} }  \frac{2^{\omega([g_1,g_2])}}{[g_1,g_2]} 
		= \prod_{i<j} \bigg(1+\frac{6}{p_i}+O\bigg(\frac{1}{p_i^2}\bigg) \bigg) 
		\le \exp\bigg\{\sum_{i<j} \frac{6}{p_i}+ O\bigg(\frac{1}{p_i^2}\bigg)\bigg\}  \ll (\log p)^6
	\]
	by  the inequality $1+t\le e^t$ and Mertens' estimate \cite[Theorem 3.4(b)]{Koukoulopoulos}. This completes the proof of part (b) of the lemma too.
\end{proof}

\subsection{Proof of Theorem \ref{thm2}}
It remains to prove Theorem \ref{thm2}. We will need the following simple consequence of Theorem 16.3 in \cite{Koukoulopoulos}:

\begin{lem}\label{lem:smooth numbers}
Let $x\ge y\ge2$ be such that $y\ge (\log x)^3$, and let $u=\log x/\log y$. Then we have that
	\[
	\sum_{\substack{d>y^u \\ P^+(d)\le y}} \frac{1}{d} \ll \frac{\log y}{u^u} .
	\]
\end{lem}

Now, let us complete the proof of Theorem \ref{thm2}. In the notation of Section \ref{sec:proofs}, we must show that if $d_1>\exp(c\log^2(s+1)/\log\log(s+2))$, then $\CA$ that does not cover $\Z$. In view of Lemma \ref{lemma1}, it suffices to show that
\[
\eta:=\sum_{1\le j\le J} \min\bigg\{ M_j^{(1)}, \frac{M_j^{(2)}}{4\delta_j(1-\delta_j)}\bigg\} <1. 
\]
Let $y=Cs^3$, where $C$ is a constant that will be chosen to be large enough, and let $k=\max\{j\in[1,J]\cap\Z: p_j\le y\}$. We set $\delta_i=0$ for $i\le k$ and $\delta_i=1/2$ for $i>k$, so that
\[
\eta\le \sum_{1\le j\le k} M_j^{(1)} + \sum_{k<j\le J} M_j^{(2)} =:\eta_1+\eta_2. 
\]
Then, Lemma \ref{lem:moments}(b) and Chebyshev's estimate \cite[Theorem 2.4]{Koukoulopoulos} imply that
\[
\eta_2 \ll \sum_{p>y} \frac{s^2(\log p)^6}{p^2} \asymp \frac{s^2(\log y)^5}{y} .
\]
If $C$ is large enough, then $\eta_2<1/2$. From now on, we fix such a choice of $C$.

It remains to bound $\eta_1$. Applying Lemma \ref{lem:moments}(a) and our assumption that 
\[
d_1>x:=\exp\{c\log^2(s+1)/\log\log(s+2)\},
\]
we find that
\[
\eta_1\le s  \sum_{\substack{d>x \\ P^+(d)\le y}} \frac{1}{d} .
\]
If $c$ is large enough compared to $C$ (which we have already fixed), then Lemma \ref{lem:smooth numbers}, applied with 
\[
u= \frac{\log x}{\log y} = \frac{c\log^2(s+1)}{\log(Cs^3)\log\log(s+2)} \sim \frac{c\log s}{\log\log s} \quad\text{when}\ s\to\infty, 
\]
implies that the sum over $d$ is $<1/(2s)$. Hence, $\eta_1<1/2$, and thus $\eta\le \eta_1+\eta_2<1$, as needed. This shows that if $d_1>x$, then $\CA$ does not cover $\Z$, thus completing the proof of Theorem \ref{thm2}. 

\subsection*{Acknowledgment}  The authors would like to thank Michael Filaseta for a helpful conversation about paper \cite{CFT}. They would also like to thank the referee of the paper for their comments and suggestions.

\subsection*{Funding} This project started during a 2020 summer internship of JK and SL, who were both funded by Undergraduate Summer Research Awards of the Natural Sciences and Engineering Research Council of Canada. 

JK is supported by a fellowship of the Fonds de recherche du Qu\'ebec - Nature et technologies.

DK is supported by the Courtois Chair II in fundamental research, by the Natural Sciences and Engineering Research Council of Canada (RGPIN-2018-05699) and by the Fonds de recherche du Qu\'ebec - Nature et technologies (2022-PR-300951).

SL is supported by a fellowship funded by the Courtois Chair II in fundamental research.

\bibliographystyle{plain}

\end{document}